\documentclass[12pt]{article}

\usepackage{amsmath}
\usepackage{amsmath,amsthm,amscd,amssymb}

\usepackage{hyperref}

\usepackage{bbm} 
\let\mathbbm\mathbf

\newcommand{\widehatT}{P}

\newcommand{\nomean}[1]{\widetilde{#1}}

\newcommand{\pnorm}[1]{\left\|#1\right\|_{{L^p}}}
\newcommand{\lipnorm}[1]{\left\|#1\right\|_{\operatorname{Lip}}}
\newcommand{\infnorm}[1]{|#1|_{{L^\infty}}}
\newcommand{\upn}{^{(n)}}

\newtheorem{thm}{Theorem}[section]

\newtheorem{prop}[thm]{Proposition}

\newtheorem{defn}[thm]{Definition}
\newtheorem{examp}[thm]{Example}

\newtheorem{rmk}[thm]{Remark}

\numberwithin{equation}{section}

\let\epsilon\varepsilon
\let\phi\varphi

\let\bar\widebar

\newcommand{\C}{\mathbb{C}}
\newcommand{\N}{\mathbb{N}}
\newcommand{\R}{\mathbb{R}}

\newcommand{\one}{\mathbf{1}}

\begin{document}

\title{ A note on large deviations for unbounded observables \thanks{The
    authors would like to thank Vaughn Climenhaga for helpful
    discussions.}}

\author{Matthew Nicol \thanks{Department of Mathematics, University of
    Houston, Houston, USA. e-mail: $\langle$nicol@math.uh.edu$\rangle$. MN
    thanks the NSF for partial support on NSF-DMS Grant 1600780.} \and
  Andrew T\"or\"ok \thanks{Department of Mathematics, University of
    Houston, Houston Texas, USA. e-mail:
    $\langle$torok@math.uh.edu$\rangle$. AT thanks the NSF for partial
    support on NSF-DMS Grant 1816315. } \thanks{Institute of Mathematics of
    the Romanian Academy, Bucharest, Romania.}}


\date{February 13, 2019} 

\maketitle

\tableofcontents

\begin{abstract}

  We consider exponential large deviations estimates for unbounded
  observables on uniformly expanding dynamical systems. We show that
  uniform expansion does not imply the existence of a rate function for
  unbounded observables no matter the tail behavior of the cumulative
  distribution function. We give examples of unbounded observables with
  exponential decay of autocorrelations, exponential decay under the
  transfer operator in each $L^p$, $1\le p < \infty$, and strictly
  stretched exponential large deviation. For observables of form
  $|\log d(x,p)|^{\alpha}$, $p$ periodic, on uniformly expanding systems we
  give the precise stretched exponential decay rate. We also show that a
  classical example in the literature of a bounded observable with
  exponential decay of autocorrelations yet with no rate function is
  degenerate as the observable is a coboundary.

\end{abstract}

\section{Introduction and background.}\label{sec:intro}


  Suppose $(T, X,\mu)$ is  a probability preserving transformation and $\varphi: X\to \mathbb R$ is a mean-zero integrable function i.e.\
$\mathbb{E}(\varphi ):=\int_{X} \varphi~d\mu =0$. 
Throughout this paper we will write $S_n (\varphi ): =\varphi +\varphi\circ T +\ldots + \varphi \circ T^{n-1} $
for the $n$th ergodic sum of $\varphi$. Sometimes we will write $S_n$ instead
of $S_n(\varphi)$ for simplicity of notation or when $\varphi$ is clear from context.
If we put $\bar{\varphi}=\int \varphi d\mu$ then ergodicity implies 
$\lim_{n\to \infty}\frac{1}{n} S_n (\varphi) =\bar{\varphi}$.



Large deviations theory concerns the rate of convergence of
$\frac{1}{n} S_n (\varphi) $ to $\bar{\varphi}$. 
%

If for each $\epsilon>0$ there exists
$C>0$ and $0<\theta<1$ such that for all $n\ge 0$
\[
\mu (|\frac{1}{n}S_n (\varphi)-\bar\varphi| > \epsilon)\le C \theta^n
\]
we will say $S_n (\varphi)$ has \emph{large deviations with an exponential
  rate}.

If for each $\epsilon>0$ there exists
$C>0$ and $0<\gamma <1$ such that for all $n\ge 0$
\[
\mu (|\frac{1}{n}S_n (\phi)-\bar\phi| > \epsilon)\le C e^{-n^{\gamma}}
\]
we say $S_n (\phi)$ has \emph{large deviations with a stretched exponential
  rate of order $\gamma>0$}.

We now recall the definition of rate function and some other notions of large deviations theory.   
\begin{defn} A mean-zero integrable function $\varphi:\Omega\to \mathbb R$
  is said to satisfy a large deviation principle with rate function
  $I(\alpha)$, if there exists a non-empty neighborhood $U$ of $0$ and a
  strictly convex function $I:U\to \mathbb R$, non-negative and vanishing
  only at $\alpha=0$, such that
\begin{eqnarray}\label{rate+}
\lim_{n\to\infty} \frac 1n\log \mu (S_n (\varphi)  \ge n \alpha)& =& -I(\alpha)
\end{eqnarray}
for all $\alpha>0$ in $U$ and
\begin{eqnarray}\label{rate-}
 \lim_{n\to\infty} \frac 1n\log \mu (S_n (\varphi)  \le n \alpha) &=& -I(\alpha)
 \end{eqnarray}
for all $\alpha<0$ in $U$.
\end{defn}

In the literature this is referred to as a first level or local (near the average) large deviations principle.


 If $\varphi$ is a mean-zero continuous observable on an SRB attractor then
 $\varphi$ has exponential large deviations (see~\cite[Theorem 2
 (2)]{LSYoung-TAMS1990}).
 For mean-zero H\"older observables on Young Towers with exponential tails
%
 (we refer to~\cite{LY98} or~\cite[Sections 2 and 4]{Nic} for the
 definition)
 which are not $L^1$ coboundaries in the sense that
 $\varphi\not = \psi\circ T-\psi$ for any $\psi\in L^1 (\mu)$) such an
 exponential large deviations result holds with rate function
 $I_{\varphi} (\alpha)$~\cite{Nic,Reybellet_Young}. A formula for the width
 of $U$ is given in \cite{Reybellet_Young} following a standard approach
 but it is not useful in concrete estimates.

\section{Erd\H{o}s-R\'enyi laws: background.}\label{erdoslaw}
Erd\H{o}s-R\'enyi laws~\cite{Erd} give estimates on the size of
time-windows over which we should expect average to achieve nontrivial
almost sure limit laws.



Proposition~\ref{prop:erdos1} below is found in a proof from Erd\H{o}s and
R\'enyi~\cite{Erd} (see \cite[Theorem 2.4.3]{Csorgo_Revesz},
Grigull~\cite{Gri}, Denker and Kabluchko~\cite{Den1} or~\cite{Denker} where
this method has been used). The Gauss bracket $[.]$ denotes the integer
part of a number. Throughout the proofs of this paper we will concentrate
on the case $\alpha>0$ as the case $\alpha<0$ is identical with the obvious
modifications of statements.

\begin{prop}[\cite{Erd,Denker}]\label{prop:erdos1}

(a) Suppose that $\varphi$ satisfies a large deviation principle with rate  function $I$ defined on the open set $U$. Let $\alpha >0$, $\alpha \in U$ and set 
$$ \ell_n=\ell_n(\alpha)=\left[\frac{\log n}{I(\alpha)}\right]\qquad n\in\mathbb N.$$
Then the upper Erd\H{o}s-R\'enyi law holds, that is, for $\mu$ a.e. $x\in X$ 
$$
 \limsup_{n\to\infty} \max_{0\le j\le n-\ell_n}\frac1{\ell_n}S_{\ell_n} (\varphi) \circ T^j (x)\le \alpha.
 $$

\noindent (b) If  for some constant $C>0$ and integer $\tau\ge 0$ for each interval $A$
\begin{eqnarray}\label{tau}
 \mu\!\left(\bigcap_{m=0}^{n-\ell_n}\{S_{\ell_n} (\varphi) \circ T^m\in A\}\right)&\le &C [\mu (S_{\ell_n}\in A)]^{n/(\ell_n)^{\tau}}
 \end{eqnarray}
then the lower Erd\H{o}s-R\'enyi law holds as well, that is, for $\mu$  a.e. $x\in X$ 
$$ 
\liminf_{n\to\infty} \max_{0\le j\le n-\ell_n}\frac1{\ell_n}S_{\ell_n} (\varphi) \circ T^j \ge \alpha.
$$
\end{prop}

\begin{rmk}
If both Assumptions (a) and (b) of Proposition~\ref{prop:erdos1} hold then 
\[
\lim_{n\to \infty}  \max_{0\le m\le n-\ell_n} \frac{S_{\ell_n}\circ T^m}{\ell_n}=\alpha.
\]
\end{rmk} 
\begin{rmk}
Note that regularity of $\varphi$ is not needed for the proof of Proposition~\ref{prop:erdos1} and that Proposition~\ref{prop:erdos1}~$(a)$ applies to unbounded observables.
\end{rmk} 



\section{Rate functions and unbounded observables.}

Alves et al~\cite{AFLV} have shown that if $T$ is a $C^{1+\delta}$ local
diffeomorphism of an interval which has stretched exponential decay of
correlations for H\"older functions versus bounded functions then there
exists a Young Tower for $T$ with stretched exponential tails. It is
unknown whether \emph{exponential} decay of correlations for H\"older
functions versus bounded functions implies there exists a Young Tower for
$T$ with \emph{exponential} tails in the same setting.

The proofs of ~\cite{Nic,Reybellet_Young} of exponential large deviations
with a rate function for a H\"older function $\varphi$ on a dynamical
system modeled by a Young Tower with exponential tails use spectral
techniques. It is necessary to establish the analyticity of the linear
operator $P_z:\mathcal{B}\to\mathcal{B}$, $z\in \C$, defined by
$P_zv=P(e^{z\phi}v)$ where $P$ is a transfer operator with spectral gap on
a Banach space of functions $\mathcal{B}$. This method of proof fails if
$\phi$ is unbounded. We remark that recently the almost sure invariance
principle has been proved for certain classes of unbounded functions on
both uniformly expanding maps and intermittent type maps with a neutral
fixed point~\cite{dedecker_gouezel_merlevede2010}. In this section we give
an example to show that we cannot expect exponential large deviations with
a rate function for unbounded observables on uniformly expanding maps, even
if they satisfy the monotonicity and moment conditions
of~\cite{dedecker_gouezel_merlevede2010} and the tails of the  cumulative distribution function
$P(\varphi >t)$ decay at any prescribed rate.



\begin{thm}\label{thm.no-exp-LD}
  Suppose $T:X\to X$ is a measure preserving map of a compact Riemannian
  manifold $(X,\mu)$. Let $p$ be a periodic point of period $\tau$. Suppose
  that $T$ is $C^1$ on an open neighborhood of the orbit of $p$. Suppose
  also there exists $\gamma>0$ such that $T^n (x)\in B_{n^{-\gamma} }(p)$
  i.o. for $\mu $ a.e. $x\in X$. Let $\varphi$ be a continuous observable
  on $X$ such that $\lim_{x\to p}\varphi (p)=\infty$, $\int \varphi d\mu=0$
  and $\varphi>-\rho$ for some $\rho>0$. Then the stationary stochastic
  process $\{\varphi\circ T^j\}$ does not satisfy exponential large
  deviations with a rate function.
\end{thm}

\begin{proof}
  Without loss of generality we take the period of $p$ to be one. If
  $\varphi$ satisfies a large deviation principle with rate function $I$
  defined on an open set $U$ then by
  Proposition~\ref{prop:erdos1} 
  (a) if $\alpha>0$ is in the interval where $I(\alpha)$ is defined and we
  let
$$ \ell_n=\ell_n(\alpha)=\left[\frac{\log n}{I(\alpha)}\right]\qquad n\in\mathbb N$$
 the upper Erd\H{o}s-R\'enyi law holds, that is, for $\mu$ a.e. $x\in X$ 
$$ \limsup_{n\to\infty} \max\{S_{\ell_n} (\varphi) \circ T^j (x)/\ell_n: 0\le j\le n-\ell_n\} \le \alpha.$$

Since $T$ is $C^1$, $|DT|_{\infty}<K$ for some $K>0$ on an open neighborhood of 
the orbit of $p$.

Fix $\alpha>0$ in $U$ and let
$M>\left[\frac{\alpha+\rho}{I(\alpha)}\right]\frac{2\log K}{\gamma}$.
Choose $N$ large enough that $\varphi (x)>M$ for all $x$ such that
$d(x,p)<\frac{1}{N^{\gamma/2}}$.
 
If $d(T^n (x),p)\le \frac{1}{n^{\gamma}}$ then
$d(T^{n+j} (x),p)\le \frac{1}{n^{\gamma/2}}$ for at least $j$ iterates,
$1\le j\le \frac{\gamma\log n}{2\log K}$ (this estimate comes from solving
$K^j\frac{1}{n^{\gamma}}=\frac{1}{n^{\gamma/2}}$). Moreover if $n>N$,
$j\le \frac{\gamma\log n}{2\log K}$ and
$T^{n+j} (x) \in [0,\frac{1}{n^{\gamma}}]$, then
$\varphi (T^{n+j} (x))\ge M$. By assumption
$T^n (x) \in B_{n^{-\gamma} (p)}$ i.o. for $\mu$ a.e. $x \in X$. If
$T^n (x) \in B_{n^{-\gamma}} (p)$ then
$S_{\ell_n} (\varphi) \circ T^j (x)>M(\gamma \frac{\log n}{2\log K})-\rho
\frac{\log n}{I(\alpha)}$ (as $\varphi \ge -\rho$). Since
$M> \left[\frac{\alpha+\rho}{I(\alpha)}\right]\frac{2\log K}{\gamma}$ this
implies that for $\mu$ a.e. $x$
\[
\limsup_{n\to\infty} \max\{S_{\ell_n} (\varphi) \circ T^j (x)/\ell_n: 0\le
j\le n-\ell_n\} >\alpha
\]
which is a contradiction to the upper Erd\H{o}s-R\'enyi law. Hence
exponential large deviations with a rate function cannot hold for this
observable.
\end{proof}

Moreover, for such observables the logarithmic moment generating function
is infinite, as shown by the next proposition.

\begin{prop}\label{thm.infinite-presure}
  Suppose $T:X\to X$ is a measure preserving map of a compact manifold
  $(X,\mu)$.
  Let $p$ be a periodic point such that
  on an open neighborhood of the orbit of $p$ the map $T$ is $C^1$, and
  there are positive constants $d$ and $c$ such that
  $\mu(\{x: d(x,p)< r\})\ge c r^d$ for small $r>0$. Let $\varphi$ be an
  observable on $X$ such that $\lim_{x\to p}\varphi (p)=\infty$,
  $\int \varphi d\mu=0$.

  Then for any $t > 0$
  \begin{equation}
    \label{eq.pressure-infinite}
    \lim_{n\to\infty} \frac{1}{n}\log\left(\int e^{t S_n(\phi)} d \mu\right) = \infty
  \end{equation}
\end{prop}

\begin{proof}

  Without loss of generality, can assume that $p$ is a fixed point.

  Pick $\lambda > 1$ such that $|T'| < \lambda$ in a neighborhood of $p$.

  Fix $t>0$; for $L>0$, let $M>L/t$ and $r>0$ such that
\[
  d(x,p)<r \implies \phi(x) > M.
\]
Since 
\[
  d(x,p)<\frac r {\lambda^n} \implies \phi(T^k(x)) > M \text{ for
    $k \le n$} \implies S_n(x) \ge n M
\]
we conclude that
\[
  \int e^{t S_n(\phi)} d \mu \ge \mu(\{x: d(x,p)<\frac r {\lambda^n}\})
  e^{t n M} \ge \frac {c r^d}{\lambda^{d n}} e^{t n M}
\]
and therefore
\begin{equation*} 
  \liminf_{n\to\infty} \frac{1}{n}\log\left(\int e^{t S_n(\phi)} d
    \mu\right) \ge t M - d \log \lambda \ge L - d \log \lambda
\end{equation*}
Since $L$ was arbitrary, we obtain~\eqref{eq.pressure-infinite}.
\end{proof}


\begin{examp}\label{Hamiltonian}

(i)  Suppose $T:X\to X$ is a $C^1$ map of a compact Riemannian metric space $X$ which preserves
 a measure $\mu$ equivalent to volume and $T$ is exponentially mixing for H\"older functions in the sense
 that for all $\phi$, $\psi$ which are $\alpha$-H\"older on $X$ there exist constants $C$, $0<\theta<1$ such that for 
 all $n\ge 0$
 \[
| \int \phi \psi\circ T^n d\mu -\int \phi d\mu \int \psi d\mu | \le C\theta^n \|\phi\|_{\alpha} \|\psi\|_{\alpha}
\]
Theorem 5.1 of \cite{HNPV} shows that if $B_i$ is a nested sequence of balls with center a periodic point $p\in X$
and there exists a constant $C_3>0$ such that $\mu (B_i)\ge C_3/i$ for all $i>0$ then
$\mu$ a.e. $x\in X$ satisfies $T^n x \in B_n$ infinitely often. Thus in this setting if
$\varphi$ is a continuous observable on $X$
  such that $\lim_{x\to p}\varphi (p)=\infty$, $\int \varphi d\mu=0$ and
  $\varphi>-\rho$ for some $\rho>0$ then   $\{\varphi\circ T^j\}$ does not satisfy exponential large deviations with
  a rate function.

   
  (ii) In particular, the tent map $T$ of the unit interval $\Omega=[0,1]$:
  \[
  T(x) = \left\{ \begin{array}{ll}
         2x & \mbox{if $0\le x \leq \frac{1}{2}$};\\
        2x-1 & \mbox{if $\frac{1}{2}\le x \le 1$}.\end{array} \right.
  \] 
  preserves Lebesgue measure and has exponential decay of correlations for
  H\"older observables on the system. The function
  $\phi (x)=\log (1-\log x) -\int \log (1-\log x) dx$ is continuous except
  at the fixed point $0$, and satisfies the assumptions of
  Theorem~\ref{thm.no-exp-LD} so does not have exponential large deviations
  with a rate function.
    \end{examp}


 
   


\subsection{Exponential decay of autocorrelations and in $L^p$, but
  strictly stretched exponential large deviations.}

In this section we consider piecewise expanding $C^2$ maps of the interval
or, more generally, Rychlik maps. For the definition of a Rychlik map
(piecewise expanding with countable many branches and additional
properties) see~\cite{Rychlik83} or~\cite[Section 5.4]{boyarsky-gora}. In
Remark~\ref{rmk.maps} we discuss further applications. We consider
observables of the form $\varphi(x)=(-\log |x-p|)^{\alpha}$ where $p$ is
periodic. We will show that $\varphi(x)$ has strictly stretched exponential
large deviations, despite having exponential decay of autocorrelations and
exponential decay in $L^p$.



Recall that for a $\mu$-preserving map $T$ which is non-singular, the
transfer operator $P:L^1(\mu)\to L^1(\mu)$ is defined uniquely by the
condition $\int (Pf) g\,d \mu =\int f (g\circ T)\,d \mu$ for all
$f \in L^1(\mu)$, $g \in L^{\infty}(\mu)$.

\begin{thm}\label{thm.slow-decay}
  Let $(T,\Omega,\mu)$ be a topologically mixing Rychlik map of the unit
  interval $\Omega=[0,1]$ and $\varphi(x):=(-\log |x-p|)^{\alpha}$ where
  $p$ is periodic for $T$ and $\alpha>0$.
%
  Assume that $T$ is $C^1$ on a neighborhood of the orbit of $p$. Then:

  (a) There are $c, C, r > 0$ such that for all $\epsilon>0$ close to zero
  and $\delta > 0$
  \[
    c\:\exp(-r n^{\frac {1}{1+\alpha}}) \le \mu \left(S_n(\phi) -n \int
      \varphi d\mu\ge n \epsilon\right) \le C \exp(-n^{\frac
      {1}{1+\alpha}-\delta}) \qquad \text{as $n\to \infty$}
  \]
  Therefore, 
  for any $\epsilon>0$ close to zero,
  \[
    \lim_{n\to \infty} \frac{\log \left[ - \log \mu \left(S_n(\phi) -n \int
          \varphi d\mu\ge n \epsilon\right)\right]}{\log n}
    =\frac{1}{1+\alpha}
  \]

  (b) $\varphi$ has exponential decay of autocorrelations: there exists
  $0<\theta<1$ and $C>0$ such that
  \[
    \left| \int (\varphi \circ T^n -\bar\phi) (\varphi -\bar\phi)\,d\mu
    \right|\le C e^{-\theta n} \qquad \text{for all $n$}.
  \]
  See Proposition~\ref{thm.autocorrelations} for a more general setting.

  (c) $\pnorm{P^n(\varphi-\bar\phi)}\to 0$ exponentially fast for each
  $p \in[1,\infty)$; see (a) of Proposition~\ref{thm.exp-decay-Lp} for
  details and more results.
\end{thm}

\begin{rmk}\label{rmk.maps}
  By equation~\eqref{eq.lower-bound}, for $p$ a fixed point it suffices in
  (a) to take $r > (\int \phi d\mu + \epsilon)^{1/\alpha}+\log |T'(p)|$.

  More generally than Rychlik maps, one can consider mixing AFN maps
  introduced in~\cite{Zwei98}. For more about these, see~\cite[Section
  1.3]{KessSch_dynsyst_arxiv}. These maps satisfy Property $\mathfrak D$
  introduced by Schindler~\cite{Schindler} and described also
  in~\cite{KessSch_dynsyst_arxiv}, see the discussion preceding
  Proposition~\ref{thm.schindler} in the Appendix.

  The mixing assumption is only used for the upper bound in (a). For the
  lower bound in (a) it suffices that $\mu$ be an a.c.i.p. whose density is
  bounded below in a neighborhood of the orbit of $p$.
\end{rmk}

\begin{rmk}
  For example consider $\varphi(x)=-\log x$, an observable defined on the
  tent map $(T,[0,1],\text{Lebesgue})$. Clearly $\varphi$ is integrable,
  $ \int\varphi dx=1$, $\varphi$ has moments of all orders and
  $\mathbb{E}[e^{t \varphi}]= \int e^{t x} e^{-x} dx$ exists for $|t|<1$.

  This satisfies the assumptions of Theorem~\ref{thm.no-exp-LD}, see
  Example~\ref{Hamiltonian} 
  of the previous section, and hence does not satisfy exponential large
  deviations with a rate function.

  However, if $X_i$ is a sequence of i.i.d. random variables with the same
  distribution function as $\varphi$ and $S_n=\sum_{j=1}^n X_j$, then
  (exponential) large deviations with a rate function holds: for
  $0<\epsilon<1$
  \[
    \lim_{n\to \infty} \frac{1}{n} \log \mathbb{P}\!\left(
      \frac{S_n}{n}>1+\epsilon\right)=-\epsilon +\log (1+\epsilon)
  \]
  This follows from~\cite[Theorem 2.6.3]
  {Durrett}; e.g., note that $\phi$ has an exponential distribution as
  $\mu (\varphi >t)=e^{-t}$ and 
  see~\cite[Example 2.6.2]{Durrett}.
\end{rmk}

\begin{rmk}
  %
  By~\cite[Proposition 4.1, part (2)]{AFLV} applied to the tent or doubling
  map on the unit interval (so their $\theta$ is 1), there exist constants
  $C_1$, $C_2$ such that for $\phi(x):=-\log|x-z|$, $\{\varphi\circ T^j\}$
  has stretched exponential deviations at rate at least as fast as
  $C_1 e^{-C_2 n^{\frac{1}{9}-\delta}}$ for any $\delta>0$.

  However, for $z$ a fixed point of the map, Theorem~\ref{thm.slow-decay}
  above with $\alpha=1$ gives an upper bound
  $C \exp(- n^{\frac{1}{2}-\delta})$ and a lower bound
  $c \exp(- [(\int \phi d\mu + \epsilon)+\log 2+\delta] n^{\frac{1}{2}})$,
  $\delta>0$.
\end{rmk}

\begin{rmk}
  

  In the setting of non-uniformly expanding maps Ara\'ujo~\cite[Theorem
  3.1]{Araujo} gives an asymptotic condition on observables of form
  $\varphi (x)=-\log |x-p|$, called exponential slow recurrence to $p$,
  which implies exponential large deviations for $\varphi$. He also proves
  that singular points for certain Lorenz-like maps with unbounded
  derivatives satisfy this condition, which is in general difficult to
  check.

  
\end{rmk}

\begin{proof}[Proof of Theorem~\ref{thm.slow-decay}] 

  We will only prove (a) here, for parts (b) and (c) see
  Proposition~\ref{thm.autocorrelations} and
  Proposition~\ref{thm.exp-decay-Lp}.

  We obtain the lower bound from an easy direct computation. For the upper
  bound we use recent results of Tanja Schindler~{\cite[Lemma
    6.15]{Schindler}}, see Proposition~\ref{thm.schindler} in the Appendix.
  Related estimates are given in~\cite{KessSch_dynsyst_arxiv}.

  We mention here that with the argument of Schindler it does not seem
  possible to obtain exponential decay of correlations by assuming
  observables have slower growth rate than $-\log^+ (x)$.

  Beyond Property $\mathfrak{D}$, we use the fact that a mixing Rychlik map
  has an invariant density in $BV[(0,1])$, bounded away from zero.


  Without loss of generality we will take $p$ to be a fixed point.

  Denote $S_n(\phi)$ by $S_n$ for short.
  Denote the absolutely continuous invariant probability of $T$ by
  $d\mu=h\,d x$, with $0< m \le h \le M$. 

  \textbf{Lower bound:}
  %
  We will discuss only the case $T'(p)>0$ and $p\not = 1$, the others being
  similar.
  Pick $\lambda >1$ such that $T'(x)\le \lambda$ in a neighborhood of $p$.

  Consider $x\in [p,p+e^{-r n^\omega}]$ with $r, \omega>0$ to be determined
  later. Then $T^j(x)\in [p,p+\lambda^j e^{-r n^\omega}]$ so
  $\phi(T^j(x))\ge [r n^\omega - j \log(\lambda)]^{1/\alpha}$ as long as
  all the iterates stay close to $p$, and therefore, for $K\le n$,
  \[
    S_n(x) \ge \sum_{j=1}^K \phi(T^j(x)) \ge K[r n^\omega - K
    \log(\lambda)]^{1/\alpha}=\widetilde{K}
  \]
  Thus
  \[
    \mu(S_n - n\bar{\phi} \ge n \epsilon) \ge \mu([p, p+e^{-r n^\omega}])
    \ge m e^{-r n^\omega}
  \]
  as long as
  \[
    \widetilde{K} \ge n(\bar{\phi}+\epsilon) \iff r n^\omega \ge
    \left(\frac{n(\bar\phi+\epsilon)}{K}\right)^{1/\alpha} + K
    \log{\lambda} \qquad \text{ and } \qquad \lambda^K e^{-rn^\omega} \ll 1
  \]
  The choice
  \begin{equation}\label{eq.lower-bound}
    \omega=\frac{1}{1+\alpha},\qquad K=n^{\omega}, \qquad r > (\bar\phi +
    \epsilon)^{1/\alpha}+\log{\lambda}
  \end{equation}
  satisfies both conditions as $n \to \infty$.
  
  \textbf{Upper bound:} We define a sequence of functions (truncated
  $\varphi$) by
  \[
    g_n(x) = \left\{ \begin{array}{ll}
                       0 & \mbox{if $ d(x,p) \leq e^{-n^{\beta}}$};\\
                       \varphi (x) & 
              \mbox{if $d(x,p) > e^{-n^{\beta}}$}.\end{array}
     \right.
  \]

  Then $\|g_n\|_{BV} \le 3 |g_n|_{\infty} =3 n^{\beta\alpha}$. We calculate
  $\mathbb{E}(g_n)=\mathbb{E}(\phi)+O(e^{-n\beta} (n\beta)^\alpha)$. Hence,
  by Schindler's estimate~\cite{Schindler} (see
  Proposition~\ref{thm.schindler} in the Appendix), if
  $G_n:=\sum_{j=1}^n g_n \circ T^j$ and $E_n:=\int G_n d\mu$ then for any
  sequence $\zeta_n $ tending to zero (we take
  $\zeta_n=(\log\log n)^{-1}$), for $n$ sufficiently large

 \[
   \mu(|G_n -E_n| >\epsilon E_n) \le 2
   \exp\!\left(-n\frac{\mathbb{E}(g_n)}{n^{\beta\alpha}}\zeta_n\right)
 \]
 \[
   = 2\exp \!\left(-n^{1-\beta\alpha}\frac{\bar\phi+ O(e^{-n\beta}
       (n\beta)^\alpha)}{\log\log n}\right) \le 2
   e^{-n^{1-\beta\alpha-\delta}}
 \]
 for $\delta > 0$ as $n\to\infty$. Taking into account that
 $\mu(S_n \not = G_n) \le O(n e^{-n^\beta})$, we choose $\beta$ to maximize
 $\min(1-\beta\alpha, \beta)$; this gives $\beta=\frac{1}{1+\alpha}$ and
 thus the claimed upper bound.
\end{proof}

We establish next exponential decay of auto correlations.

\begin{rmk}
  Since for the above maps $T$ the transfer operator $P$ has a spectral gap
  on $BV$, it follows easily that for any $z\in [0,1]$ the function
  $\phi(x):=(-\log d(x,z))^\alpha$, $\alpha > 0$, has exponential decay
  against $BV$-observables.
\end{rmk}

\begin{prop}\label{thm.autocorrelations}
  Assume $T:[0,1]\to[0,1]$ has an a.c.i.p. $\mu$ with density bounded
  above, and its transfer operator $P$ has a spectral gap on $BV$.
  Then $\phi(x):=(-\log d(x,z))^\alpha$, $\alpha > 0$, has exponential
  decay of autocorrelations.

  Namely, let $r$ be the spectral radius of $P$ on
  $BV\cap \{f \mid \int f d\mu = 0 \}$. Then, for every
  $\beta <\min \{-\log{ r }, \frac{1}{2}\}$ there exists a constant
  $C_{\beta} >0$ such that
\[
  \left| \int (\varphi \circ T^n -\bar\phi) (\varphi -\bar\phi)\,d\mu
  \right|\le C_\beta e^{-\beta n} \qquad \text{for all $n$}.
\]
\end{prop}

\begin{proof}
  Define a sequence of truncations
\[
  h_n(x) := \min\{M_n, \phi(x)\},
  \qquad \delta_n:=\phi-h_n, \qquad \bar{\phi}:=\int \phi d\mu, \qquad
  \bar{h_n}:=\int h_n d\mu
\]
with $M_n$ to be determined later. Note that $\phi\in L^2(\mu)$ since $\mu$
has a density bounded above.

Obviously $\|h_n\|_{BV}\le 3 |h_n|_{\infty} = 3 M_n$. Using the spectral
gap of the transfer operator $P$ in $BV$, for every $r < \rho < 1$ there is
a $C_P>0$ such that
$|P^n(h)|_\infty \le \|P^n(h)\|_{BV} \le C_P \rho^n \|h\|_{BV}$ provided
$\int h d\mu=0$.
%

Then, using the spectral gap of $P$ on BV, H\"older's inequality and that
$|\bar{h_n}-\bar \phi|=|\int \delta_n d\mu| \le \|\delta_n\|_{L^2}$, and
denoting by $C$ all constants that do not change with $M_n$ and $n$:
\begin{align*}
  \left| \int \phi\circ T^n \cdot (\phi-\bar \phi) d\mu \right| 
  &= \left|
    \int \phi\circ T^n \cdot [(h_n-\bar {h_n}) + (\delta_n + \bar
    {h_n}-\bar \phi)d\mu \right| 
  \\
  & \le \left| \int \phi \cdot P^n(h_n-\bar {h_n}) d\mu\right|  +
    \|\phi\|_{L^2} (\|\delta_n\|_{L^2} + |\bar
    {h_n}-\bar \phi|) \\
  & \le C  \rho^n \|\phi\|_{L^1} \|h_n-\bar {h_n}\|_{BV} + C \|\phi\|_{L^2} \|\delta_n\|_{L^2}
\end{align*}
Since 
\[
  \delta_n(x)\not = 0 \iff d(x,z) \le M_n^{1/\alpha}
\]
and $\mu$ has a density bounded above, for any $\delta >0$ we obtain
\[
  \|\delta_n\|_{L^2}^2 \le C\int_{|x-z|\le M_n^{1/\alpha}}
  |\log(|x-p|)|^{2\alpha} d x \le C_\delta e^{-(1-\delta) M_n^{1/\alpha}}
\]
and therefore
\[
  \left| \int \phi\circ T^n \cdot (\phi-\bar \phi) d\mu \right|  \le C
  \|\phi\|_{L^1} M_n  \rho^n + C C_\delta \|\phi\|_{L^2}  e^{-\frac{1}{2}(1-\delta) M_n^{1/\alpha}}
\]
Choose now $M_n=n^\alpha$.
%
%
Hence for every $\beta <\min \{-\log{\rho}, \frac{1}{2}\}$ there exists a
constant $C_{\beta} >0$ such that for all $n$
\[
  \left| \int (\varphi \circ T^n -\bar\phi) (\varphi -\bar\phi)\,d\mu
  \right|\le C_\beta e^{-\beta n}.
\]
\end{proof}

\subsection{Stretched exponential large deviations from exponential decay
  in Lipschitz}



In this section we slightly improve estimates of  \cite{AFLV}[Proposition 4.1 (b)] to obtain 
a better stretched exponent decay rate.  

\begin{thm}\label{th.LD-from-Lip}
  Let $(T, X, \mu)$ with $X=[0,1]$ be a dynamical system where $\mu$ is a
  $T$-invariant a.c.i.p having density bounded above,
  $d\;\mu/d\;{\operatorname{Lebesgue}}\le C_\mu$.

  Consider the observation $\phi:[0,1]\to \R$, $\phi(x):=|\log d(x,z)|$ for
  some $z\in X$.

  Assume $P$, the transfer operator w.r.t. $\mu$, has exponential decay in
  the space of Lipschitz functions: there are constants $\theta, C_P>0$
  such that
  \[
    \lipnorm{P^n f} \le C_P e^{-\theta n} \lipnorm{f} \text{ provided } \int f d
    \mu =0.
  \]
  
  Fix $\alpha<1/5$. Then for $\epsilon > 0$ close to zero there are
  $C=C_{\epsilon, \alpha}>0$ and $r=r_{\epsilon, \alpha}>0$ such that
  \[
    \mu\left(\left|S_n(\phi -\int \phi d \mu)\right|> n \epsilon\right) \le
    C \exp(- r n^{\alpha})
  \]
\end{thm}
\begin{rmk}\ 
  \begin{enumerate}
  \item The doubling map satisfies the hypothesis of
    Theorem~\ref{th.LD-from-Lip}. Therefore, by
    Theorem~\ref{thm.slow-decay} cannot have decay rate faster than
    $\exp(-n^{1/2})$.

  \item In this setting a similar stretched exponential large deviations
    was also obtained in \cite{AFLV}[Proposition 4.1 (b)], but with rate
     $\exp(-n^{1/9})$.
  \end{enumerate}
\end{rmk}

\begin{proof}
  For simplicity, we take $z$ to be zero, so $\phi=|\log x|$.

  We denote by 
  \[
    \nomean{f} := f-\int f d\mu
  \]
  and when this correction $\int f d\mu$ is $O(1)$ we might ignore it in
  the estimates.

  We introduce three parameters, $M_n$, $C_n$ and $\alpha$, whose optimal
  value will be determined at the end.

  Denote $h\upn:=\min\{\phi, M_n\}$.
  Define $w\upn:=\sum_{k\ge 1} P^k \nomean{h\upn}$ and write
  $h\upn :=g\upn+w\upn\circ T- w\upn$. Since $P g\upn=0$, its Birkhoff sums
  form a martingale.

  Recall the Azuma-Hoeffding inequality: if
  $S_n:=\sum_{k=1}^n X_n$ is a martingale
  whose increments $X_n$ satisfy $|X_n|\le M_n$ then
  \[
    P(S_n \ge A)\le \exp\left({-\frac{A^2}{2\sum_{i=1}^n M_i^2}}\right)
  \]

  Using Azuma-Hoeffding,
    \begin{align*}\nonumber
    \mu(S_n(\nomean{h\upn})> 3 n\epsilon) & \le \mu(S_n(\nomean{g\upn})>
    n\epsilon) + 2 \mu(w\upn > n\epsilon) \\ &
      \le \exp(-\frac{n^2 \epsilon^2}{2
      n \infnorm{\nomean{g\upn}}^2}) + 2 \mu(w\upn > n\epsilon)
    \end{align*}
    and therefore 
    \begin{align}\nonumber
    \mu(S_n(\nomean{\phi})> 3 n\epsilon) & \le \mu(S_n(\nomean{h\upn})>
    3 n\epsilon) + n \mu(h\upn \not= \phi) \\\label{eq.LD-estimate_lip}
    &
      \le \exp(-\frac{n^2 \epsilon^2}{2
      n \infnorm{\nomean{g\upn}}^2}) + 2 \mu(w\upn > n\epsilon) + C_\mu n e^{-M_n}
    \end{align}

  Compute 
  \begin{equation}\label{eq.estimate_gn}
    \infnorm{\nomean{g\upn}} \le  \infnorm{h\upn}+2 \infnorm{w\upn}
  \end{equation}
  and, using the exponential decay of $P$ and that
  $\infnorm{P f}\le \infnorm{f}$,
  \begin{align}\nonumber
    \infnorm{w\upn} & =\infnorm{\sum_{k=1}^\infty P^k \nomean{h\upn}}\\\nonumber
                    & \le \sum_{k=1}^{C_n} \infnorm{P^k \nomean{h\upn}} +
                      \sum_{k=C_n+1}^\infty \lipnorm{P^k \nomean{h\upn}} \\\nonumber
                    & \le C_n M_n + C_P \sum_{k=C_n+1}^\infty e^{-\theta
                      k}\lipnorm{\nomean{h\upn}} \\\nonumber
                    & \le C_n M_n + C' e^{-\theta
                      C_n}\lipnorm{\nomean{h\upn}} \\\label{eq.estimate_wn}
                    & \le C_n M_n + C' e^{-\theta C_n} e^{M_n}
\end{align}


We want to bound each term in \eqref{eq.LD-estimate_lip} by
$\exp(-n^\alpha)$.
For the first term this requires
$\infnorm{\nomean{g\upn}}\approx n^{(1-\alpha)/2}$, hence, in view of
\eqref{eq.estimate_gn} and \eqref{eq.estimate_wn}, need
\begin{equation}\label{eq.constrant1}
  C_n M_n \approx n^{(1-\alpha)/2}\qquad - \theta C_n + M_n \lesssim
  \frac{1-\alpha}{2} \log n
\end{equation}
which lead to
\begin{equation*}
  M_n=n^{(1-\alpha)/4}\qquad C_n=\frac{1}{\theta} M_n =\frac{1}{\theta} n^{(1-\alpha)/4}
\end{equation*}

Then the second term in \eqref{eq.LD-estimate_lip} is zero because
$\infnorm{w\upn}\lesssim n^{(1-\alpha)/2}$, and the third term becomes
$n e^{-M_n}\approx \exp(-n^{(1-\alpha)/4})$.

Thus the best $\alpha$ comes from
$\max\min\{n^\alpha,n^{(1-\alpha)/4}\}=n^{1/5}$ for $\alpha=1/5$.
\end{proof}

\subsection{Exponential decay in $L^p$ for all $p\ge1$ does not imply
  exponential large deviations}

We now show that exponential decay of the transfer operator in each $L^p$
does not imply exponential large deviations. 


The next result is stated for the doubling map, but the proof applies, with
the appropriate changes, to any map having an a.c.i.p. with density bounded
above, and whose transfer operator has a spectral gap, e.g. on $BV$.

\begin{prop}\label{thm.exp-decay-Lp}
  Let $(T,X,\mu)$ be the doubling map on $X=[0,1]$, $P$ its transfer
  operator with respect to the Lebesgue measure. Denote by $e^{-\theta}$
  the exponential decay rate of its transfer operator on $BV$ functions of
  mean zero, and by $C_T$ the constant that is involved (see
  \eqref{eq.decay} for the precise meaning of $\theta$ and $C_T$).
  \begin{enumerate}
  \item[(a)] Let $\phi_1 (x):=(-\log x)^{\alpha}$, $\alpha >0$. Then for
    all $p\ge 1$, there is a constant $C_{\alpha p}$ such that
    \[
      \text{$\left\|P^n \left(\phi_1 -\int \phi_1 d x\right)\right\|_{L^p}
        \le e^{-\theta n} (C_T M_n+C_{\alpha p} M_n^\alpha)$, where
        $M_n=n p \theta$.}
    \]
    In particular, the decay is exponential with rate $e^{-\beta}$ for any
    $\beta<\theta$. The constant $C_{\alpha p}$ comes from
    \eqref{eq.log-alpha}.

  \item[(b)] Let $\phi_2(x):=x^{-\alpha}$ with $0< \alpha < 1$. Then for
    $1\le p < 1/\alpha$
    \[
      \text{
        $\left\|P^n \left(\phi_2 -\int \phi_2 d x\right)\right\|_{L^p}
        \lesssim e^{- n \theta(1-\alpha p)}$ as $n\to\infty$}
    \]
  \end{enumerate}
\end{prop}

\begin{proof}

  \newcommand{\psio}{\nomean{\psi}}

  The transfer operator $P$ is a contraction in each $L^p$, $p \ge 1$, and
  has exponential decay on functions in $BV$ with mean zero; so can find
  $C_T, \theta>0$ such that for $f\in L^1$ \footnote{These inequalities
    also hold when $f$ is not in $L^p$ or $BV$, in that case the RHS is
    infinity.}
 \begin{equation}\label{eq.decay}
   \|P f \|_p \le \|f\|_p, \qquad
   \left\|P^n \left(\phi -\int \phi d x\right)\right\|_{BV} 
   \le C_T e^{-\theta n} \left\|P^n \left(\phi -\int \phi d x\right)\right\|_{BV} 
\end{equation}

For either $\phi_k$, $k=1, 2$, substract for simplicity its mean, so denote
$\psio:=\phi_k -m_k$ where $m_k=\int \phi_k d x$, and let
$\psi_n:=\min\{\psio, M_n\}$; in what follows we will ignore this constant
$m_k$, it does not change much.

 Then $\|\psi_n\|_{BV}=M_n$ (more precisely, $M_n+m_k$), and
 \begin{equation}
   \begin{aligned}\nonumber  
     \|P^n\psio\|_p & \le \|P^n\psi_n\|_p + \|P^n(\psio-\psi_n)\|_p    \\
     &\le \|P^n\psi_n\|_{BV} + \|\psio-\psi_n\|_p \le C_T e^{-\theta n}
     \|\psi_n\|_{BV} + \|\psio-\psi_n\|_p \\
     & \le C_T e^{-\theta n} M_n + \|\psio-\psi_n\|_p
   \end{aligned}
 \end{equation}

 Now compute that for $\phi_1 (x):=(-\log x)^\alpha$, $\alpha>0$ (ignoring
 again $m_1$):
 \begin{align}\nonumber  
     \|\psio-\psi_n\|_p & = \left(\int_{\phi_1>M_n} |\phi_1-M_n|^p d x
     \right)^{1/p} \le \left(\int_{0 < x < e^{-M_n}} |\log(x)|^{\alpha p} d
       x
     \right)^{1/p}\\
     \label{eq.log-alpha}
     & \le C_{\alpha p} \left( x |\log(x)|^{\alpha p}
       \Big\vert^{x=e^{-M_n}}_{x=0} d x \right)^{1/p} = C_{\alpha p}
     e^{-M_n/p} M_n^{\alpha}
 \end{align}
   Take $M_n= n p \theta$ to get the desired decay in $L^p$.

   Doing the same for $\phi_2 (x):=x^{-\alpha}$, $0<\alpha<1$,
   $1\le p < 1/\alpha$:
 \begin{equation}
   \begin{aligned}\nonumber  
     \|\psio-\psi_n\|_p & \le \left(\int_{0 < x < M_n^{-1/\alpha}}
       x^{-\alpha p} d x \right)^{1/p} = \frac{1}{(1-\alpha p)^{1/p}}
     M_n^{1-\frac {1}{\alpha p}}
   \end{aligned}
 \end{equation}
 Take $M_n=\exp({ p \alpha \theta n})$ to get exponential decay at rate
 $e^{-\theta (1-\alpha p)}$.
\end{proof}

\section{Exponential large deviations without a rate function for bounded
  observables.}

Examples exist in the literature~\cite{Bradley,Orey_Pelikan,Bryc, Chung} of
stationary processes which have exponential large deviations but a rate
function does not exist. In particular there is an example of a mean zero
bounded function $f$ taking only $3$ values on an aperiodic recurrent
Markov chain $(X_n)$ with a countable state space such that the system has
exponential large deviations but does not have a rate function. In this
example defining $S_n=\sum_{j=0}^{n-1} f(X_j)$ for all $\epsilon >0$, there
exist constants $C(\epsilon)$, $0<\gamma<1$ such that
$\mathbb{P}( |\frac{S_n}{n}|>\epsilon)\le C(\epsilon)e^{-\gamma n}$, giving
exponential convergence in the strong law of large numbers yet there is no
rate function controlling the rate of decay. We show in the next section
that in these examples $f$ is a coboundary and there exists
$\psi \in L^2 (\mathbb{P})$ such that $f(X_j)=\psi (X_{j+1})-\psi (X_j)$
for all $j\ge 0$ so that the example is degenerate (the variance
$\sigma^2=0$). The assumption that the observable is not a coboundary (and
hence that $\sigma^2>0$) is made in the statements of the theorems establishing large deviations
with rate functions in~\cite{Nic}.

According to Bryc and Smolenski~\cite{Bryc} the idea of the example was due
to
\break
Bradley~\cite{Bradley} and also adapted by Orey and
Pelikan~\cite[Example 4.1]{Orey_Pelikan}. 
\break
Bradley~\cite{Bradley} produced
an example of a stationary, pairwise independent, absolutely regular
stochastic process for which the central limit theorem does not hold. Orey
and Pelikan presented this system as an example of a strongly mixing shift
for which the large deviation principle with rate function fails. Bryc and
Smolenski showed that there is in fact also an exponential convergence in
the strong law of large numbers. Bryc and Smolenski's work was recast by
Chung~\cite{Chung} into dynamical systems language, and the system was
expressed as a Young Tower $(F,\Delta,\nu)$. We will also recast as a
dynamical system and show that $f$ is a coboundary, in fact
$f=\psi\circ F-\psi$ where $\psi$ is unbounded but $\psi\in L^2$. This
seems to have been overlooked in the literature. In fact if $f$ were not a
coboundary the example would contradict results of
\cite{Nic,Reybellet_Young} which imply that any bounded Lipschitz function
on a Young Tower with exponential tails which is not a coboundary has
exponential large deviations with a rate function. In Section~\ref{Bryc} we
also present the example in a dynamical setting following Bryc and
Smolenski's notation and overall presentation and give an explicit
coboundary for $f$. As far as we know there is no example of a bounded
observable on a dynamical system, with non-zero variance, which has
exponential large deviations and yet no rate function.

\subsection{The example of Bryc and Smolenski}\label{Bryc}

 Let $\Delta_0$ (which for concreteness we will identify
with the unit interval $[0,1]$) be the base of a Young Tower  $\Delta$ with $\Delta_0$  partitioned into intervals $\Lambda_0, \Lambda_1,\ldots,
\Lambda_k...,$. Let $n(0)=1/2$, $n(k)=12^k$, $p_k=Ce^{-\frac{n(k)}{2}}$ where $C^{-1}=\sum_{k=0}^{\infty}e^{-\frac{n(k)}{2}}$ is a normalization constant to 
ensure the Tower has a probability measure. Each $\Lambda_k$ will have Lebesgue  measure $p_k$ and  we define the 
Tower return time function  on $\Lambda_k$ as $R_{\Lambda_k}:=R(k)=2n(k)$. We now build the Tower
$\Delta$ above the base. We write $\Lambda_{k,0}:=\Lambda_k$ and define, for $0\le j\le R(k)-1$ the levels $\Lambda_{k,j}$ of the Tower
lying above $\Lambda_k$ by
\[
\Delta = \bigcup_{k\in{N}^+, 0\le j \le R_k-1}\{ (x, j) : x\in \Lambda_{0,k }\}
\] 
 with the  tower map $F:\Delta \to \Delta $ given by
\[
F(x, j) = \begin{cases}(x, j+1) & \mbox{ if }  x\in \Lambda_{k,0}, j<R(k)-1\\
(T_k x, 0) &\mbox{ if } x\in \Lambda_{k,0}, j = R(k)-1\end{cases}.
\]
where $T_k$ has constant derivative and maps the interval $\Lambda_{k,0}$ bijectively onto
$\Delta_0$.  We define $F$ on $\Lambda_{0,0}$ by requiring that $F$ map  $\Lambda_{0,0}$ bijectively onto
$\Delta_0$. Note that $m(\Lambda_{0,0})=Ce^{-1/4}$. This requirement on the height $R(0)$ of $\Lambda_{0,0}$ to be $1$ is to
ensure aperiodicity.

We lift Lebesgue measure $m$ from the base to the Tower to obtain a measure $\tilde{\nu}$ on $\Delta$ and then
define a probability measure $\nu =\tilde{\nu}$ by normalization. The map $F$ preserves $\nu$ and is exponentially mixing for  a Banach space of
observables on $\Delta$~\cite{LY98}.

If $k\not =0$ we define $f: \lambda_{k,j} \to \{-1,0,1\}$ by
\[
f (x, j) = \begin{cases}1
 & \mbox{ if }  x\in \Lambda_{k}, j \le n(k)-1\\
-1 &\mbox{ if } x\in \Lambda_{k}, n(k) \le j \le 2n(k)-1\end{cases}.
\]
if $k=0$ we take  $f(0,0)=0$. This is the example model of~\cite{Bradley, Orey_Pelikan, Bryc, Chung}.

Now define a function $\psi$, which will be a coboundary for $f$, by 
\[
\psi  (x, j) = \begin{cases}j
 & \mbox{ if }  x\in \Lambda_{k},0\le  j \le n(k)\\
2n(k)-j &\mbox{ if } x\in \Lambda_{k}, n(k) < j \le 2n(k)-1\end{cases}.
\]
and we take $\psi(0,0)=0$.

It is easy to check that
 \[
f=\psi\circ F -\psi
\] 

Hence $S_n (f)=\psi\circ F^n -\psi$.  We remark that, as $f=\psi\circ F-\psi$ is a  coboundary and $\psi\in L^2$, $S_n (f)$ has zero-variance and does not satisfy a non-trivial central limit theorem.

Since 
\[
\nu (\psi\circ F^n >n)=\nu (\psi >n) \sim e^{-n/2}
\]
 it is  easy to see that $\nu (|S_n/n |>\epsilon)$ decays exponentially for any $\epsilon$ (and obtain explicit estimates).

In~\cite[Theorem 2.2]{Varadhan} Varadhan shows that large deviations with rate function fails if 
$\frac{1}{n} \log \mathbb{E}(\exp(S_n))$ does not have a limit as 
$n\to \infty$ and it is easy to show (see also~\cite{Bryc}) that $\frac{1}{n} \log \mathbb{E}(e^{S_n})$ 
has no limit in this example. 

For example consider points in $x\in \Lambda_{k,0}$ where
$\nu(\Lambda_{k,0})=Ce^{-\frac{n(k)}{2}}$. Then
$S_{n(k)}(x)=\psi\circ F^{n(k)}=n(k)$ so that $e^{S_{n(k)} (x)}= e^{n(k)}$.
Hence
\[
\frac{1}{n(k)} \log \mathbb{E}(e^{S_{n(k)}}))= \frac{1}{2n(k)}\log
(Ce^{-\frac{n(k)}{2}}e^{n(k)})=\frac{1}{4}+\frac{\log C}{n(k)}.
\]
Thus $\limsup \frac{1}{n} \log \mathbb{E}(e^{S_n})\ge \frac{1}{4}$. Similar
considerations show that
$\liminf \frac{1}{n} \log \mathbb{E}(e^{S_n})\le 0$ and consequently
$\frac{1}{n} \log \mathbb{E}(e^{S_n})$ has no limit.

\section{Discussion.}

There are a number of open yet natural questions. 

(1)  If $p$ is a non-periodic point does the observable $\varphi(x)=-\log d(x,p)$  on  a uniformly expanding map of the interval satisfy exponential large deviations? Does exponential large deviations  for  $\varphi(x)=-\log d(x,p)$
hold for a full measure set of $p$? 

(2) Are there natural conditions on the decay rate of the  tail of the distribution $\mu (\varphi>t)$ which imply
exponential large deviations if $\varphi$ is an unbounded observable on a uniformly expanding map? Theorem 3.1 shows that we cannot expect to have exponential large deviations with a rate function.

(3) Is there an example of a non-degenerate bounded observable $\varphi$,
$\varphi\not =\psi-\psi\circ T$, on a (smooth) dynamical system which
satisfies exponential large deviations but does not have a rate function?

\section{Appendix}

We describe Schindler's result that we use, {\cite[Lemma 6.15]{Schindler}},
which states that Property $\mathfrak{D}$ for
$(\Omega, \mathcal{B}, T, \mu, \mathcal{F}, \|\cdot\|, \chi)$ implies
Property \textbf{A} for $(\chi\circ T^{n-1})_n$, with {Property
  $\mathfrak{D}$} defined in~\cite[Definition 1.7 in Section
1.2]{Schindler} and Property \textbf{A} in~\cite[Definition
2.2]{Schindler}.
We state this result only for the situation we are interested in, namely
the Banach space $\mathcal{F}$ being $BV([0,1)]$. We will use
$\|f\|_{BV}:=\operatorname{var}(f)+|f|_\infty$ with $\operatorname{var}$
denoting the total variation seminorm.

For a function $\chi:\Omega\to \mathbb{R}$ and $\ell\in\mathbb{R}$, denote
its truncation by
$\chi^{\ell}:=\chi\cdot\mathbbm{1}_{\left\{\chi\leq \ell\right\}}$.

\begin{prop}[{\cite[Lemma 6.15]{Schindler}}]\label{thm.schindler}
  Let $T:[0,1]\to [0,1]$ be a transformation and $\mu$ a $T$-invariant
  probability measure for which $T$ is non-singular; denote by $P$ the
  transfer operator associated to $T$ with respect to the measure $\mu$.
  Assume:
  \begin{itemize}
  \item[(a)] the $T$-invariant probability measure $\mu$ is mixing;

  \item[(b)] $\widehatT$ is bounded on $BV[(0,1)]$ and has a spectral gap
    on $BV$ (that is, the spectral radius of $\widehatT$ restricted to the
    codimension one subspace $BV\cap \{f \mid \int f d\mu = 0 \}$ is less
    than 1);

  \item[(c)] $\chi\in L^1(\mu)$ with $\chi\ge 0$ and
    $|\chi|_{L^\infty} = \infty$;

  \item[(d)] there exists $C>0$ such that for all $\ell>0$,
    \begin{align}\nonumber
      \left\|\chi^{\ell}\right\|_{BV}\leq C\cdot \ell
    \end{align}
    where $X^f:=X\cdot \one_{X\le f}$.
  \end{itemize}
  
  Then,
  for every sequence $(\xi_n)_{n\in\N}$ tending monotonically to zero and
  for every $\epsilon>0$, there exists $N\in \N$ such that for every
  positive valued sequence $(f_m)$:
  \[
    n\ge N \implies \mathbb{P}\left(\left|T^{f_n}_n
        -\mathbb{E}\left(T^{f_n}_n\right)\right| >\epsilon
      \mathbb{E}\left(T^{f_n}_n\right)\right) \le 2
    \exp\!\left(-\xi_n\frac{\mathbb{E}\left(T^{f_n}_n\right)}{f_n}\right)
  \]
  where $T^{f_n}_n:=\sum_{k=1}^n \chi^{f_n} \circ T^k$ is the Birkhoff sum
  of $n$ truncated terms.

  [Note that if $\chi\ge 0$ is bounded then we get this estimate with
  $f_n= |\chi|_{L^\infty(\mu)}$.]
\end{prop}



\providecommand{\MR}[1]{}
\providecommand{\bysame}{\leavevmode\hbox to3em{\hrulefill}\thinspace}
\providecommand{\MR}{\relax\ifhmode\unskip\space\fi MR }
\providecommand{\MRhref}[2]{%
  \href{http://www.ams.org/mathscinet-getitem?mr=#1}{#2}
}
\providecommand{\href}[2]{#2}

\end{document}